\documentclass[11pt]{amsart}
\usepackage{amsmath,graphics}
\usepackage{amsfonts,amssymb}
\usepackage{xypic}
\usepackage[colorlinks,plainpages,urlcolor=blue]{hyperref}

\newtheorem{theorem}{Theorem}[section]
\newtheorem{lemma}[theorem]{Lemma}
\newtheorem{corollary}[theorem]{Corollary}
\newtheorem{proposition}[theorem]{Proposition}

\theoremstyle{definition}
\newtheorem{definition}[theorem]{Definition}
\newtheorem{example}[theorem]{Example}
\newtheorem{remark}[theorem]{Remark}
\newtheorem{question}[theorem]{Question}

\newtheorem{conjecture}[theorem]{Conjecture}

\newtheorem*{ack}{Acknowledgement}

\newcommand{\V}{\mathcal{V}}

\def\a{\alpha}

\def\l{\lambda}

\def\Q{\mathbb{Q}}

\def\Z{\mathbb{Z}}
\def\R{\mathbb{R}}
\def\C{\mathbb{C}}
\def\N{\mathbb{N}}

\def\CP{\mathbb{CP}}

\def\zt{\Z[t^{\pm 1}]}
 
\def\bp{\begin{pmatrix}}\def\ep{\end{pmatrix}}
\def\bn{\begin{enumerate}}\def\en{\end{enumerate}}
\def\be{\begin{equation}} \def\ee{\end{equation}}
\def\ba{\begin{array}} \def\ea{\end{array}}
\def\sm{\setminus}

\def\wti{\widetilde}

\def\qk{quasi-K\"ahler}

\def\qp{quasi-projective }
\def\qsp{quasi-projective}
\def\op{\operatorname}
\def\th{\op{th}}

\newcommand{\cv}{\check{{\mathcal V}}}
\newcommand{\BT}{\mathcal{B}_T}

\DeclareMathOperator{\im}{im}

\DeclareMathOperator{\rank}{rank}
\DeclareMathOperator{\ord}{ord}
\DeclareMathOperator{\Hom}{Hom}

\DeclareMathOperator{\const}{const}

\DeclareMathOperator{\tor}{Tors}
\DeclareMathOperator{\Newt}{Newt}

\def\set#1{{\{ #1\}}}
\newcommand{\surj}{\twoheadrightarrow}

\newcommand{\abs}[1]{\left| #1 \right|}

\begin{document}

\title[K\"{a}hler groups, quasi-projective groups, and $3$-manifold groups]
{K\"{a}hler groups, quasi-projective groups, and \\ $3$-manifold groups}

\author[Stefan Friedl]{Stefan Friedl}
\address{Mathematisches Institut, Universit\"at zu K\"oln, Germany}
\email{sfriedl@gmail.com}

\author[Alexander Suciu]{Alexander I. Suciu$^1$}
\address{Department of Mathematics, Northeastern University,
Boston, MA 02115, USA}
\email{a.suciu@neu.edu}

\thanks{$^1$Partially supported by NSF grant DMS--1010298}



\date{\today}
\begin{abstract}
We prove two results relating $3$-manifold groups
to fundamental groups occurring in complex geometry.
Let $N$ be a compact, connected, orientable
$3$-manifold.  If $N$ has non-empty, toroidal
boundary, and $\pi_1(N)$ is a K\"ahler group,
then $N$ is the product of a torus with an interval.
On the other hand, if $N$ has either empty or
toroidal boundary, and $\pi_1(N)$ is a quasi-projective
group, then all the prime components of $N$ are
graph manifolds.
\end{abstract}
\maketitle

\section{Introduction and main results}
\label{sect:intro}

\subsection{}
A classical problem, going back to J.-P.~Serre, asks for a
characterization of fundamental groups of smooth projective
varieties.  Still comparatively little is known about this class
of groups.  For instance, it is not known whether it coincides
with the putatively larger class of groups that can be realized
as fundamental groups of compact K\"ahler manifolds, known
for short as K\"ahler groups.

Another, very much studied class of groups is that consisting
of fundamental groups of compact, connected, $3$-dimensional
manifolds, known for short as $3$-manifold groups.
Recent years have seen the complete validation of the
Thurston program for understanding $3$-manifolds.
This effort, begun with the proof of the Geometrization Conjecture
by Perelman, has culminated in the results of  Agol \cite{Ag12},
Wise \cite{Wi12a} and Przytycki--Wise \cite{PW12} which now
give us a remarkably good understanding of $3$-manifold
groups (see also \cite{AFW12} for more information).

In this context, a natural question arises: Which
$3$-manifold groups are K\"{a}hler groups?
This question, first raised by S.~Donaldson
and W.~Goldman in 1989, and independently
by Reznikov in 1993 (see \cite{Re02}), has led
to a flurry of activity in recent years.
In \cite{DS09} the second author and A.~Dimca
showed that if the fundamental group of a closed $3$-manifold
is K\"ahler, then the group is finite.  Alternative proofs have
since been given by Kotschick \cite{Kot12} and by
Biswas, Mj and Seshadri \cite{BMS12}, while the
analogous question for quasi-K\"{a}hler groups
was considered in \cite{DPS11}.

\subsection{}
In this paper, we pursue these lines of inquiry in two directions.
First, we determine which fundamental groups of $3$-manifolds
with non-empty, toroidal  boundary are K\"ahler.
Note that the $2$-torus is a K\"ahler manifold. We will show
that, in fact, $\Z^2$ is the only fundamental group of a
$3$-manifold with non-empty, toroidal boundary which
is also a K\"ahler group.  More precisely, we will prove
the following theorem.

\begin{theorem}
\label{mainthmk}
Let $N$ be a $3$-manifold with non-empty, toroidal boundary.
If $\pi_1(N)$ is a K\"ahler group, then
$N\cong S^1\times S^1\times [0,1]$.
\end{theorem}

If $N$ is allowed to have non-toroidal boundary components,
other K\"ahler groups can appear.   For instance, if $\Sigma_g$
is a Riemann surface of genus $g\ge 2$, then $\pi_1(\Sigma_g\times [0,1])$
is certainly a K\"ahler group.  The complete, full-generality classification of
 K\"ahler $3$-manifold groups is still unknown to us.

\subsection{}
Next, we turn to the question, which $3$-manifold groups are \qsp,
that is, occur as fundamental groups of complements of divisors
in a smooth, complex projective variety.   Examples of $3$-manifolds
with fundamental groups that are \qp are given by the exteriors
of torus knots and Hopf links, by connected sums of $S^1\times S^2$'s,
and by Brieskorn manifolds; we refer to \cite[\S1.2]{DPS11}
for more examples.  Note that all these examples are
connected sums of graph manifolds.

We will show that this is not a coincidence:

\begin{theorem}
\label{mainthmqk}
Let $N$ be a $3$-manifold with empty or toroidal boundary.
If $\pi_1(N)$ is a \qp group, then all the prime components
of $N$ are graph manifolds.
\end{theorem}

In particular, the fundamental group of a hyperbolic $3$-manifold
with empty or toroidal boundary is never a \qp group.

In fact, as we shall see in \S\S\ref{section:nonprimeqk} and
\ref{sect:questions}, we can get even finer results.
For instance, if a $3$-manifold $N$ as above is not prime,
then none of its prime components is virtually fibered.
But, at the moment, the complete determination
of which $3$-manifold groups are \qp groups is still
out of reach for us.

\subsection{}

The proofs of Theorems \ref{mainthmk} and  \ref{mainthmqk}
rely on some deep results from both complex geometry and
$3$-dimensional topology.  We extract from those results
precise qualitative statements about the Alexander polynomials
of the manifolds that occur in those settings. Comparing the
answers gives us enough information to whittle down
the list of fundamental groups that may occur in the
opposite settings until we reach the desired conclusions.

Results by Arapura \cite{Ar97}, as strengthened in \cite{DPS09}
and \cite{ACM10}, put strong restrictions on the characteristic
varieties of both K\"ahler manifolds and smooth, quasi-projective
varieties.  Following the approach from \cite{DPS08}, we use
those restrictions to conclude that the respective Alexander
polynomials are ``thin":  their Newton polytopes have
dimension at most $1$.

Recent results of Agol \cite{Ag08,Ag12},
Kahn--Markovic \cite{KM12}, Przytycki--Wise \cite{PW12}
and Wise \cite{Wi12a,Wi12b} give us a good understanding
of fundamental groups of irreducible $3$-manifolds which
are not graph manifolds. Using those results, we show
that every irreducible $3$-manifold with empty or toroidal
boundary is either not a graph manifold, or it admits a finite
cover which has a ``thick'' Alexander polynomial.

We conclude this paper with a short list of open questions.

\subsection*{Convention.}
All groups are understood to be finitely presented and all
manifolds are understood to be compact, orientable and
connected, unless otherwise stated.

\section{Alexander polynomials, characteristic varieties, and thickness}
\label{section:alexpoly}

\subsection{Orders of modules}
\label{section:orders}

Let $H$ be a finitely generated, free abelian group and let
$M$ be a finitely generated module over the group ring $\Z[H]$.
Since the ring $\Z[H]$ is Noetherian, there exists a finite presentation
\begin{equation}
\label{eq:pres}
\xymatrix{\Z[H]^r \ar^{\a}[r]& \Z[H]^s \ar[r]& M \ar[r]& 0}.
\end{equation}
We can furthermore arrange that $r\geq s$, by adding zero
columns if necessary.

Given an integer $k\geq 0$, the
{\em $k$-th elementary ideal}\/ of $M$, denoted by
$E_k(M)$, is the ideal in $\Z[H]$ generated by all
minors of size $s-k$ of the matrix $\a$. The {\em $k$-th order}\/
of the module $M$ is a generator of the smallest principal ideal
in $\Z[H]$ containing $E_k(M)$; that is,
\begin{equation}
\label{eq:ordk}
\ord^k_{\Z[H]}(M)= \text{$\gcd$ of all $(s-k)\times (s-k)$-minors of $\a$}.
\end{equation}

Note that the element $\ord^k_{\Z[H]}(M)\in \Z[H]$ is
well-defined up to multiplication by a unit in $\Z[H]$;
if the ring $\Z[H]$ is understood, then we will just write $\ord^k(M)$.
Furthermore, we will write $\ord_{\Z[H]}(M):=\ord^0_{\Z[H]}(M)$.

Denote by $\tor_{\Z[H]}M$ the $\Z[H]$-torsion submodule of $M$.
Note that $\ord^0(M)\ne 0$ if and only if $M$ is a torsion
$\Z[H]$-module (see e.g.~\cite[Remark~4.5]{Tu01}).
Furthermore, denote by $r$ the rank of $M$ as a $\Z[H]$-module.
It then follows from \cite[Lemma~4.9]{Tu01} that
\begin{equation}
\label{equ:tu01}
\ord^i(M)=
\begin{cases}
0 &\text{if $i<r$}, \\[2pt]
\ord^{r-i}\big(\tor_{\Z[H]}M \big) &\text{if $i\geq r$}.
\end{cases}
\end{equation}

\subsection{The thickness of a module}
\label{section:thick}

As before, let $H$ be a finitely generated, free abelian group,
and let $M$ be a finitely generated $\Z[H]$-module.  Write
\begin{equation}
\label{eq:ord}
\ord_{\Z[H]}(\tor_{\Z[H]}M)=\sum_{h\in H}a_hh.
\end{equation}

\begin{definition}
\label{def:thick}
The \emph{thickness}\/ of the $\Z[H]$-module $M$ is
the integer
\begin{equation}
\label{eq:thick}
\th_{\Z[H]}(M):=\dim \op{span} \{ g-h\in H\otimes \Q
\mid a_g\ne 0\text{ and }a_h\ne 0\}.
\end{equation}
\end{definition}

Put differently, the thickness $\th(M):=\th_{\Z[H]}(M)$ is the
dimension of the Newton polyhedron of the Laurent polynomial
$\ord_{\Z[H]}(\tor_{\Z[H]}M)$.  Note that the definition does
not depend on a representative for this polynomial.
Later on we will make use of the following lemma.

\begin{lemma}
\label{lem:thicknessadds}
Let $H_1, \dots, H_r$ be free abelian groups, and let $M_i$
be finitely generated modules over $\Z[H_i]$.
Set $H:=\bigoplus_{i=1}^{r} H_i$, and view $\Z[H_i]$
as subrings of $\Z[H]$.  Then for $i=1,\dots,r$ we have
\[
\ord_{\Z[H]}(\tor_{\Z[H]}(M_i\otimes_{\Z[H_i]}\Z[H]))=
\ord_{\Z[H_i]}\tor_{\Z[H_i]}M_i,
\]
and furthermore
\[
\th_{\Z[H]}\Big(\bigoplus M_i\otimes_{\Z[H_i]}\Z[H]\Big)=
\sum \th_{\Z[H_i]}(M_i).
\]
\end{lemma}

\begin{proof}
It follows easily from the definitions that
\[
\tor_{\Z[H]}(M_i\otimes_{\Z[H_i]}\Z[H])=(\tor_{\Z[H_i]}M_i)\otimes_{\Z[H_i]}\Z[H].
\]
In particular,
\[
\ord_{\Z[H]}(\tor_{\Z[H]}(M_i\otimes_{\Z[H_i]}\Z[H]))=
\ord_{\Z[H_i]}\tor_{\Z[H_i]}M_i,
\]
and so
\[
\ord_{\Z[H]} \tor_{\Z[H]} \Big(\bigoplus M_i\otimes_{\Z[H_i]}\Z[H]\Big)
= \prod \ord_{\Z[H_i]}\tor_{\Z[H_i]}M_i.
\]
The desired conclusion follows at once.
\end{proof}

\subsection{Alexander polynomials}
\label{subsec:alex poly}

Let $X$ be a connected CW-complex with finitely many $1$-cells.
We denote by $\wti{X}$ the universal cover of $X$. Note that $\pi_1(X)$
canonically acts on $\wti{X}$ on the left; we use the natural involution
$g\mapsto g^{-1}$ on $\pi_1(X)$ to endow $\wti{X}$ with a right
$\pi_1(X)$-action.

Let $H:=H_1(X;\Z)/\tor$ be the maximal torsion-free
abelian quotient of $\pi_1(X)$.  We view $\Z[H]$ as a left
$\pi_1(X)$-module via the canonical projection
$\pi_1(X)\surj H$. Consider the tensor product
$C_*(\wti{X})\otimes_{\Z[\pi_1(X)]} \Z[H]$.
This defines  a chain complex of $\Z[H]$--modules, and we denote
its homology groups by $H_*(X;\Z[H])$. Most important for
our purposes is the {\em Alexander invariant}, $A_X= H_1(X;\Z[H])$.

For each integer $k\ge 0$, we define the {\em $k$-th
Alexander polynomial}\/ of $X$ as
\begin{equation}
\label{eq:delx}
\Delta_X^k:=\ord^k_{\Z[H]}(A_X).
\end{equation}
The Laurent polynomial $\Delta_{X}^k\in \Z[H]$ is
well-defined up to multiplication by a unit in $\Z[H]$,
and only depends on $\pi_1(X)$.  We write
$\Delta_X:=\Delta_X^0$, and call it the
{\em Alexander polynomial}\/ of $X$. If $\pi$ is a
finitely generated group, then we denote by
$\Delta_{\pi}^k$ the Alexander polynomials of its
Eilenberg--MacLane space.  Note that
$\Delta_X^k = \Delta_{\pi_1(X)}^k$.

We denote the thickness of the module $A_X$
by $\th(X)$. It follows from the definitions and
formula \eqref{equ:tu01} that
\begin{equation}
\label{eq:thx}
\th(X)=\dim( \Newt( \Delta_X^r)),
\end{equation}
where $r=\rank_{\Z[H]} (A_X)$.  In particular, if
the Alexander invariant is a torsion $\Z[H]$-module, then
$\th(X)=\dim( \Newt( \Delta_X))$.

Let $\C^{*}$ be the multiplicative group of non-zero
complex numbers. We shall call the connected algebraic group
$\widehat{H}=\Hom(H,\C^{*})$ the {\em character torus}\/
of $X$. The Laurent polynomial $\Delta_{X}^k$ can be
viewed as a regular function on
$\widehat{H}$. As such, it defines
a hypersurface,
\begin{equation}
\label{eq:vdelta}
V(\Delta^k_X) = \set{\rho \in \widehat{H} \mid \Delta^k_X(\rho)=0}.
\end{equation}

For instance, if $K$ is a knot in $S^3$, with exterior
$X=S^3\setminus \nu K$, then $\Delta_X$ is the classical
Alexander polynomial of the knot, and $V(\Delta^k_X)\subset \C^*$
is the set of roots of $\Delta_X$, of multiplicity at least $k$.

\subsection{Homology jump loci}
\label{subsec:cvs}

The {\em characteristic varieties}\/ of $X$ are
the jump loci for homology with coefficients in the rank
$1$ local systems defined by characters inside the
character torus of $X$.  For each $k\ge 1$, the set
\begin{equation}
\label{eq:cvs}
\V_k(X)=\{\rho \in \widehat{H}
\mid \dim H_1(X,\C_{\rho})\ge k\}
\end{equation}
is a Zariski closed subset of $\widehat{H}$.  These varieties
depend only on the group $\pi=\pi_1(X)$, so we will sometimes
write them as $\V_k(\pi)$.  For more details on all this, we
refer to \cite{Su11, Su12}.

As shown by Hironaka \cite{Hi97}, the characteristic varieties
coincide with the varieties defined by the Alexander ideals of $X$,
at least away from the trivial representation.  More precisely,
\begin{equation}
\label{eq:ve}
 \V_k(X) \setminus \set{1} =V(E_{k-1}(A_X)) \setminus \set{1} .
 \end{equation}

The next lemma details the relationship between the
hypersurfaces defined by the Alexander polynomials of
$X$ and the characteristic varieties of $X$.  (The case
$k=1$ was proved by similar methods in \cite[Corollary 3.2]{DPS08}.)

\begin{lemma}
\label{lem:delta cv}
For each $k\ge 1$, let $\cv_k(X)$ be the union of all
codimension-one irreducible components of $\V_k(X)$.
Then,
\begin{enumerate}
\item \label{dc1}
$\Delta^{k-1}_X=0$ if and only if
$\V_k(X)=\widehat{H}$, in which
case $\cv_k(X)=\emptyset$.
\item \label{dc2}
If $b_1(X)\ge 1$ and $\Delta^{k-1}_X\ne 0$, then
\begin{equation*}
\cv_k(X) =\begin{cases}
V(\Delta^{k-1}_X) & \text{if $b_1(X)\ge 2$}\\
V(\Delta^{k-1}_X)\coprod \{ 1\}  & \text{if $b_1(X)=1$}.
\end{cases}
\end{equation*}
\end{enumerate}
\end{lemma}

\begin{proof}
Given an ideal $\mathfrak{a}\subset \Z[H]$,
let $\check{V}(\mathfrak{a})$ be the union of all
codimension-one irreducible components of the
subvariety $V(\mathfrak{a}) \subset \widehat{H}$
defined by $\mathfrak{a}$.
As noted in \cite[Lemma 3.1]{DPS08},
we have that $V(\gcd(\mathfrak{a})) = \check{V}(\mathfrak{a})$.

Applying this observation to the ideal $\mathfrak{a}=E_{k-1}(A_X)$,
and using formula \eqref{eq:ve}, we see that $V(\Delta^{k-1}_X)=\cv_k(X)$,
at least away from the identity.  The desired conclusions follow at once.
\end{proof}

The next theorem generalizes Proposition 3.7 from \cite{DPS08},
which treats the case $k=1$ along the same lines.  For the sake
of completeness, we provide full details.

\begin{theorem}
\label{thm:parallel}
Suppose $b_1(X)\ge 2$.  Then $\Delta^{k-1}_X\doteq \const $ if
and only if $\cv_k(X)=\emptyset$; otherwise, the following are equivalent:
\begin{enumerate}
\item \label{p1}
The Newton polytope of $\Delta^{k-1}_X$ is a line segment.
\item \label{p2}
All irreducible components of $\cv_k(X)$ are
parallel, codimension-one subtori of $\widehat{H}$.
\end{enumerate}
\end{theorem}

\begin{proof}
The first equivalence follows at once from Lemma \ref{lem:delta cv}.
So let us assume $\Delta:=\Delta^{k-1}_X$ is non-constant, and
set $n=b_1(X)$.

First suppose \eqref{p1} holds.  Then, in a suitable coordinate system
$(t_1,\dots ,t_n)$ on $\widehat{H}=(\C^*)^n$,  the polynomial
$\Delta$ can be written as $(t_1-z_1)^{\alpha_1}\dots (t_1-z_n)^{\alpha_n}$,
for some pairwise distinct, non-zero complex numbers $z_i$
and positive exponents $\alpha_i$. From Lemma \ref{lem:delta cv},
we conclude that $\cv_k (X)$ is the (disjoint) union of the parallel
subtori $\{ t_1=z_1 \}, \dots  , \{ t_1=z_n \}$.

Next, suppose \eqref{p2} holds.  Then again in a suitable
coordinate system, we have that $\cv_k (X)= \bigcup_i\{ t_1=z_i \}$.
Now let $\Delta = f_1^{\beta_1}\cdots f_q^{\beta_q}$ be the
decomposition of $\Delta$ into irreducible factors.  Then
$V(\Delta)$ decomposes into irreducible components
as $\bigcup_j \{ f_j= 0 \}$.  Since the two decompositions into
irreducible components of $\cv_k (X)=V(\Delta)$ must agree,
we must have that $\Delta \doteq \prod_i (t_1-z_i)^{\alpha_i}$.
Hence, $\Newt(\Delta)$ is a line segment, and we are done.
\end{proof}

\section{K\"ahler groups}
\label{sect:kahler}

A {\em K\"{a}hler manifold}\/ is a compact, connected,
complex manifold without boundary, admitting a Hermitian
metric $h$ for which the imaginary part $\omega=\im(h)$
is a closed $2$-form.    The class of K\"{a}hler manifolds,
is closed under finite direct products. The main source
of examples are smooth, complex projective varieties, such
as Riemann surfaces, complex Grassmannians,
and abelian varieties.

Now suppose $\pi$ is a K\"ahler group, i.e., there is a
K\"ahler manifold $M$ with $\pi=\pi_1(M)$.
This condition puts severe restrictions on the group $\pi$,
besides the obvious fact that $\pi$ must be finitely presented
(we refer to \cite{ABCKT96} for a comprehensive survey).

For instance, the first Betti number $b_1(\pi)$ must be even,
and all higher-order Massey products of classes in $H^1(\pi, \Q)$
vanish.  Furthermore, the group $\pi$ cannot split as a non-trivial
free product, by work of Gromov \cite{Gr89} and Arapura,
Bressler, and Ramachandran \cite{ABR92}.  Finally,
as shown in \cite{DPS09}, the only right-angled Artin groups
which are also K\"ahler groups are the free abelian groups
of even rank.

The pull-back of a K\"ahler metric to a finite cover is again
a K\"ahler metric.  It follows that the  finite cover of a K\"ahler
manifold is also a K\"ahler manifold.  We thus have the following
lemma (see also \cite[Example~1.10]{ABCKT96}.)

\begin{lemma}
\label{lem:kahlersubgroup}
Any  finite-index subgroup of a K\"ahler group is again
a K\"ahler group.
\end{lemma}

An analogous result holds for fundamental groups
of smooth, complex projective varieties.

The basic structure of the characteristic varieties of
K\"{a}hler manifolds was described in work of Beauville,
Green--Lazarsfeld, Simpson, Campana, and Arapura \cite{Ar97}.
We state a simplified version of this result, in the
form we need it.

\begin{theorem}
\label{thm:cv kahler}
Let $M$ be a K\"{a}hler manifold.  Then, for each $k\ge 1$,
all the positive-dimensional, irreducible components
of $\V_{k}(M)$ are even-di\-men\-sional
subtori of the character torus of $M$,
possibly translated by torsion characters.
\end{theorem}

The Alexander polynomial of a K\"ahler group is
highly restricted.   The next result sharpens
Theorem~4.3(3) from \cite{DPS08}, where a
similar result is proved in the case when $\pi$ is
the fundamental group of a smooth projective
variety, and $k=0$.

\begin{theorem}
\label{thm:deltakahler}
If $\pi$ be a  K\"ahler group.  Then, for any $k\ge 0$, the
polynomial $\Delta_\pi^k$ is a constant.  In particular, $\th(\pi)=0$.
\end{theorem}

\begin{proof}
From Hodge theory, we know that $b_1(\pi)$ is even.
If $b_1(\pi)=0$, there is nothing to prove; so we may as
well assume $b_1(\pi)\ge 2$.

Now, from Theorem \ref{thm:cv kahler} we know that
all positive-dimensional irreducible components of $\V_{k+1}(\pi)$
are even-dimensional.  Thus, there are no codimension-one
components in $\V_{k+1}(\pi)$; in other words, $\cv_{k+1}(\pi)=\emptyset$.
Finally, Theorem \ref{thm:parallel} implies that
$\Delta_\pi^k$ is constant.
\end{proof}

\section{Quasi-projective groups}
\label{sect:qp}

A manifold $X$ is said to be a {\em (smooth) quasi-projective
variety}\/ if there is a connected, smooth, complex projective
variety $\overline{X}$ and a divisor $D$ such that
$X=\overline{X}\setminus D$.   (If $\overline{X}$ is
only known to admit a K\"{a}hler metric, then $X$
is said to be a {\em quasi-K\"{a}hler manifold}.)
Using resolution of singularities, one may choose
the compactification $\overline{X}$ so that
$D=\overline{X}\setminus X$ is a normal-crossings
divisor.  An important source of examples is provided
by complements of hypersurfaces in $\CP^n$.

Now suppose $\pi$ is a \qp group, i.e., there is a smooth,
\qp variety $X$ such that $\pi=\pi_1(X)$.  Again, the group
$\pi$ must be finitely presented, but much weaker restrictions
are now imposed on $\pi$ than in the case of K\"ahler groups.
For instance, $b_1(\pi)$ can be arbitrary, non-trivial
Massey products can occur, and $\pi$ can split as a
non-trivial free product.

Examples of \qp groups include all finitely generated
free groups $F_n$, which may realized as
$\pi_1(\mathbb{CP}^1\setminus \{\text{$n+1$ points}\})$,
and all free abelian groups $\Z^n$, which may be realized
as $\pi_1((\C^*)^n)$.  In fact, a right-angled Artin group $\pi$
is \qp if and only $\pi$ is a direct product of free groups
(possibly infinite cyclic), see \cite{DPS09}.

The analog of Lemma \ref{lem:kahlersubgroup} holds for
\qp groups.  Though this is presumably folklore, we
could not find an explicit reference in the literature,
so we include a proof, kindly supplied to us
by Donu Arapura.

\begin{lemma}
\label{lem:qpsubgroup}
Let $\pi$ be a group, and let $\wti{\pi}$ be a finite-index subgroup.
If $\pi$ is \qsp, then $\wti{\pi}$ is \qp as well.
\end{lemma}

\begin{proof}
Let $X$ be a smooth \qp variety, and let $Y\to X$ be a finite
cover. Choose a projective compactification $\overline{X}$
of $X$, and normalize $\overline{X}$ in the function field of
$Y$ to get $\overline{Y}$.  Then $\overline{Y}$ is a normal
variety containing $Y$ as an open subset.  Moreover,
$\overline{Y}$ is a projective variety, since the pullback
of an ample line bundle under the finite cover
$\overline{Y} \to \overline{X}$ is again ample. Thus,
$Y$ is a smooth \qp variety, and this finishes the proof.
\end{proof}

The analogous (and more difficult) result for
quasi-K\"ahler groups is proved in \cite[Lemma~4.1]{AN99}.

The basic structure of the cohomology support loci of
smooth, \qp varieties was established by Arapura \cite{Ar97}.
Additional information on the nature of these varieties
has been provided in work of Dimca \cite{Di07}, Dimca,
Papadima and Suciu \cite{DPS08, DPS09},  Artal-Bartolo,
Cogolludo and Matei \cite{ACM10}, and most recently,
by Budur and Wang \cite{BW12}.  The next theorem
summarizes some of those known results, in the form
needed here.

\begin{theorem}
\label{thm:arapura}
Let $X$ be a smooth, \qp variety, and set $H=H_1(X,\Z)/\tor$.
Then, for each $k\ge 1$, the following hold:
\begin{enumerate}
\item \label{a1}
Every irreducible component of
$\V_k(X)$ is of the form $\rho T$, where $T$ is an
algebraic subtorus of $\widehat{H}$, and $\rho$ is
a torsion element in $\widehat{H}$.

\item \label{a2}
If $\rho_1 T_1$ and $\rho_2 T_2$ are two such components,
then either $T_1=T_2$, or $T_1\cap T_2$ is finite.
\end{enumerate}
\end{theorem}

\begin{proof}
Statement \eqref{a1} is proved in \cite{Ar97} for the wider
class of quasi-K\"ahler manifolds $X$, but only for
positive-dimensional irreducible components:
in that generality, the isolated points in $\V_k(X)$ are only
known to be unitary characters. Now, if $X$ is a smooth \qp
variety, it is shown in \cite[Theorem 1]{ACM10}, and also
in \cite[Theorem 1.1]{BW12}, that the isolated points in
$\V_k(X)$ are, in fact, torsion points.

Statement \eqref{a2} is proved in \cite[Theorem 4.2]{DPS08},
for $k=1$; the general case is established in
\cite[Proposition 6.5]{ACM10}.
\end{proof}

The intersection theory of translated subtori in a
complex algebraic torus was worked out by
E.~Hironaka in \cite{Hi96}, and was further
developed in \cite{Na09, SYZ13}.    In particular,
the following lemma holds.

\begin{lemma}
\label{lem:tt int}
Let $T_1$ and $T_2$ be two algebraic subtori in
$(\C^{*})^n$, and let $\rho_1$ and $\rho_2$ be two
elements in $(\C^{*})^n$.  Then
$\rho_1 T_1 \cap \rho_2 T_2\ne\emptyset$ if and
only if $\rho^{}_1 \rho_2^{-1}\in T_1  T_2$,
in which case $\dim (\rho_1 T_1 \cap \rho_2 T_2) =
\dim(T_1\cap T_2)$.
\end{lemma}

In view of this lemma, Theorem \ref{thm:arapura} has
the following immediate corollary.

\begin{corollary}
\label{cor:ttk}
Let $\pi$ be a \qp group.  Then, for
each $k\ge 1$, the irreducible components of
$\V_{k}(\pi)$ are (possibly torsion-translated) subtori
of the character torus $\widehat{H}$.  Furthermore,
any two distinct components of $\V_{k}(\pi)$ meet
in at most finitely many (torsion) points.
\end{corollary}

The Alexander polynomial of a \qp group must satisfy
certain rather restrictive conditions, though not as
stringent as in the K\"ahler case.  The next result
sharpens Theorem~4.3 from \cite{DPS08}, while
following a similar approach.

As before, given a finitely-generated group $\pi$, let
$H=H_1(\pi;\Z)/\tor$, written multiplicatively.
Given a polynomial $p(t)=\sum_{i\in\Z}a_it^i \in \zt$,
and an element $h\in H$, we put
$p(h):=\sum_{i\in \Z}a_ih^i\in \Z[H]$.

\begin{theorem}
\label{thm:deltaquasikahler}
Let $\pi$ be a \qp group, and assume $b_1(\pi)\ne 2$.
Then, for any $k\ge 0$, the following hold.

\begin{enumerate}
\item \label{q1}
The polynomial $\Delta_\pi^k$ is either zero,
or the Newton polytope of $\Delta^k_\pi$ is
a point or a line segment. In particular, $\th(\pi)\leq 1$.

\item \label{q2}
There exists a polynomial $p(t)\in \zt$ of the form
$c\cdot q(t)$, where $c\in \Z\sm \{0\}$ and $q(t)$ is a product
of cyclotomic polynomials, and an element $h\in H$,
such that $\Delta^k_\pi=p(h)$.
\end{enumerate}
\end{theorem}

\begin{proof}
We start with part \eqref{q1}.  Set $n=b_1(\pi)$.  If either
$n=0$ or $1$ or $\Delta_\pi^k$ is zero, there is nothing
to prove. So we may as well assume that $n\ge 3$, and
$\Delta_\pi^k\ne 0$.

From Corollary \ref{cor:ttk}, we know that all
irreducible components of $\V_{k+1}(\pi)$ are
(possibly translated) subtori of $\widehat{H}=(\C^*)^n$,
meeting in at most finitely many points.
Suppose $\rho_1 T_1$ and $\rho_2 T_2$ are
two components of codimension $1$, that
are not parallel. Then, by Lemma \ref{lem:tt int},
\begin{align}
\label{eq:dim}
\dim(\rho_1 T_1\cap \rho_2 T_2) &= \dim(T_1\cap T_2) \\
&=\dim(T_1)+\dim(T_2)-\dim(T_1T_2)= n-2 \ge 1, \notag
\end{align}
a contradiction.  Thus, all codimension-one components of
$\V_{k+1}(\pi)$ must be parallel subtori.  Claim \eqref{q1} then
follows from Theorem \ref{thm:parallel}.

We now prove part \eqref{q2}.  Using the previous part, we
may find a polynomial $p\in \Z[t]$ such that, after a change of
variables if necessary, $\Delta^k_{\pi}(t_1,\dots, t_n)=p(t_1)$.
Clearly, $V(\Delta^k_{\pi}) =V(p)\times (\C^*)^{n-1}$. On the
other hand, by Lemma \ref{lem:delta cv} and Theorem \ref{thm:arapura},
the hypersurface $V(\Delta^k_{\pi})$ is a (possibly torsion-translated)
subtorus in $(\C^*)^n$.   Thus, all roots of $p$ must be roots of
unity, and claim \eqref{q2} follows.
\end{proof}

\section{Three-manifold groups}
\label{section:thurston}

\subsection{The Thurston norm and fibered classes}
\label{subsection:thurston}

Throughout this section, $N$ will be a $3$-manifold with either
empty or toroidal boundary.  Given a  surface $\Sigma$ with
connected components $\Sigma_1, \dots , \Sigma_s$, we put
$\chi_{-}(\Sigma)=\sum_{i=1}^{s} \max\{-\chi(\Sigma_i),0\}$.

Let $\phi\in H^1(N;\Z) = \Hom(\pi_1(N),\Z)$ be a non-trivial
cohomology class.
We say that $\phi$ is a \emph{fibered class}\/ if there exists
a fibration $p\colon N\to S^1$ such that the induced map
$p_*\colon \pi_1(N)\to \pi_1(S^1)=\Z$ coincides with $\phi$.
It is well-known that $\phi$ is a fibered class if and only if
$n\phi$ is a fibered class, for any non-zero integer $n$.
We now say that $\phi\in H^1(N;\Q)\sm \{0\}$ is fibered if
it is a rational multiple of a fibered integral class.

The \emph{Thurston norm}\/ of  $\phi\in H^1(N;\Z)$  is defined as
\begin{equation}
\label{eq:th norm}
\| \phi \|_T=\min \big\{ \chi_-(\Sigma) \mid \text{$\Sigma$ a properly
embedded surface in $N$, dual to $\phi$} \big\}.
\end{equation}
W.~Thurston \cite{Th86} (see also \cite[Chapter~10]{CC03})
proved the following results:
\bn
\item \label{th1}
$\|-\|_T$ defines a norm%
\footnote{Throughout this paper, we allow norms to be
degenerate; that is, we do not distinguish between the
notions of norm and seminorm.}
on $H^1(N;\Z)$, which can be extended to a norm
$\|-\|_T$ on $H^1(N;\Q)$.
\item \label{th2}
The unit norm ball, $B_T=\{\phi\in H^1(N;\Q) \mid \|\phi \|_T\leq 1\}$,
is a rational polyhedron with finitely many sides.
\item \label{th3}
There exist open, top-dimensional faces $F_1,\dots,F_k$
of the Thurston norm ball such that
$\{\phi\in H^1(N;\Q) \mid  \text{$\phi$ fibered}\}=
\bigcup_{i=1}^k \Q^+ F_i$.
 \en

Thus, the set of fibered classes form a cone on certain open,
top-dimensional faces of $B_T$, which we refer to as
the \emph{fibered faces}\/ of the Thurston norm ball.
The polyhedron $B_T$ is evidently symmetric in the origin. We say that
two faces $F$ and $G$ are \emph{equivalent}\/ if $F=\pm G$.
Note that a face $F$ is fibered if and only if $-F$ is fibered.

The Thurston norm is degenerate in general, for instance, for
$3$-manifolds with homologically essential tori.  On the other
hand, the Thurston norm of a hyperbolic $3$-manifold is
non-degenerate, since a hyperbolic $3$-manifold admits
no homologically essential surfaces of non-negative
Euler characteristic.

We denote by $\BT$ the dual of the Thurston norm ball
of $N$, that is,
\begin{equation}
\label{eq:dual ball}
\BT:=\{ p\in H_1(N;\Q) \mid \text{$\phi(p)\leq 1$ for all
$\phi\in H^1(N;\Q)$ with $\|\phi\|_T\leq 1$} \}.
\end{equation}

From the above discussion, we know that $\BT$ is a compact
convex polyhedron in $H_1(N,\Q)$. Its vertices are canonically
in one-to-one correspondence with the top-dimensional faces
of the Thurston norm ball  $B_T$.  (Thurston \cite{Th86}
showed that these vertices correspond in fact to integral
classes in $H_1(N;\Z)/\tor\subset H_1(N;\Q)$.) We say
that a vertex $v$ of $\BT$ is fibered if it corresponds to
a fibered face of $B_T$.

\subsection{Quasi-fibered classes}
\label{subsec:qf}
The Thurston norm and the set of fibered classes `behave well'
under going to finite covers.
More precisely, the following proposition holds:

\begin{proposition}
\label{prop:tnfinitecover}
Let $p\colon N'\to N$ be a finite cover and let $\phi\in H^1(N;\Q)$.
Then $\phi$ is fibered if and only if $p^*(\phi)\in H^1(N';\Q)$ is fibered.
Furthermore,
\[
\| p^*(\phi) \|_T=[N':N]\cdot \|\phi\|_T.
\]
In particular, the map $p^*\colon H^1(N;\Q)\to H^1(N';\Q)$ is,
up to the scale factor $[N':N]$, an isometry which maps fibered cones
into fibered cones.
\end{proposition}

\begin{proof}
The first statement is an immediate consequence of Stallings'
fibration theorem \cite{St62}, while the second statement follows
from work of Gabai \cite{Ga83}.
\end{proof}

For future purposes we also introduce the following definition:
We say that a class $\phi\in H^1(N;\Q)$ is \emph{quasi-fibered}\/ if $\phi$
lies on the boundary of a fibered cone of the Thurston norm ball of $N$.
Note that $\phi$ is quasi-fibered if and only if $\phi$  is the limit of
a sequence of fibered classes in $H^1(N;\Q)$.
Proposition \ref{prop:tnfinitecover}, then, has the
following immediate corollary.

\begin{corollary}
\label{cor:pull qf}
Let $p\colon N'\to N$ be a finite cover.
\begin{enumerate}
\item \label{qf1}
If $\phi\in H^1(N;\Q)$ is a quasi-fibered class, then $p^*(\phi)\in H^1(N';\Q)$
is also quasi-fibered.
\item \label{qf2}
Pull-backs of inequivalent faces of the Thurston norm ball of
$N$ lie on inequivalent faces of the Thurston norm ball of $N'$.
\end{enumerate}
\end{corollary}

\subsection{Thurston norm and Alexander norm}
\label{subsec:thurston}

As before, let $N$  be a $3$-manifold with empty or toroidal boundary.
In \cite{Mc02}, McMullen defined the {\em Alexander norm}\/
on $H^1(N;\Q)$, as follows.

Let $\Delta_N\in \Z[H]$ be the Alexander polynomial of $N$,
where $H=H_1(N;\Z)/\tor$.  We write
\begin{equation}
\label{eq:delta}
\Delta_{N}=\sum_{h\in H}a_hh.
\end{equation}

Let $\phi \in H^1(N;\Q)$.  If $\Delta_N=0$, then we define
$\|\phi\|_A=0$.  Otherwise, the Alexander norm of $\phi$
is defined as
\begin{equation}
\label{eq:a-norm}
\|\phi\|_A:=\max\,\{ \phi(a_h)-\phi(a_g)\mid \text{$g,h\in H$ with
$a_g\ne 0$ and $a_h\ne 0$}\}.
\end{equation}
This evidently defines a norm on $H^1(N;\Q)$.  We now have the following theorem:

\begin{theorem}
\label{thm:mcmullen}
Let $N$  be a  $3$-manifold
with empty or toroidal boundary and such that $b_1(N)\geq 2$. Then
\be
\label{equ:mcm}
\|\phi\|_A\leq \|\phi\|_T \text{ for any }\phi\in H^1(N;\Q).
\ee
Furthermore, equality holds for any quasi-fibered class.
\end{theorem}

McMullen \cite{Mc02} proved inequality \eqref{equ:mcm},
and he also showed that \eqref{equ:mcm} is an equality for
fibered classes.  It now follows from the continuity of the
Alexander norm and the Thurston norm that \eqref{equ:mcm}
is an equality for quasi-fibered classes.

The following is now a well-known consequence of McMullen's theorem:

\begin{corollary}
\label{cor:mcmullen}
Let $N$  be a $3$-manifold with empty or toroidal boundary.
\bn
\item \label{thick1}
If $N$ fibers over the circle, and if the fiber has
negative Euler characteristic, then $\th(N)\geq 1$.
\item \label{thick2}
If $N$ has at least two non-equivalent fibered faces,
then $\th(N)\geq 2$.
\en
\end{corollary}

\begin{proof}
As usual, we write $H=H_1(N;\Z)/\tor$ and we
view $H$ as a subgroup of $H_1(N;\Q)=H\otimes \Q$.  We write
$\Delta_{N}=\sum_{h\in H}a_hh$, and we set
\begin{equation}
\label{eq:hull}
C:=\op{hull}\{h-g\in H\otimes \Q \mid \text{$a_g\ne 0$ and
$a_h\ne 0$}\}.
\end{equation}
Theorem \ref{thm:mcmullen} now says that $C$ is a subset
of $\BT$, and that any fibered vertex of $\BT$ also belongs to $C$.
The desired conclusions follow at once.
\end{proof}

\subsection{The RFRS property}
\label{section:akmw}

We conclude this section with two remarkable theorems that have
revolutionized our understanding of $3$-manifold groups.  First,
a definition which is due to Agol \cite{Ag08}.

\begin{definition}
\label{def:rfrs}
A group $\pi$ is called \emph{residually finite rationally solvable (RFRS)}\/
if there is a filtration of groups $\pi=\pi_0\supset \pi_1 \supset
\pi_2\supset \cdots $ such that the following conditions hold:
\bn
\item $\bigcap_i \pi_i=\{1\}$.
\item  For any $i$, the group $\pi_i$ is a normal, finite-index
subgroup of  $\pi$.
\item For any $i$, the map $\pi_i\to \pi_i/\pi_{i+1}$ factors
through $\pi_i\to H_1(\pi_i;\Z)/\tor$.
\en
\end{definition}

We say that a group $\pi$ is {\em virtually RFRS}\/ if $\pi$
admits a finite-index subgroup which is RFRS.  The following
theorem is due to Agol \cite{Ag08} (see also \cite{FK12}).

\begin{theorem}
\label{thm:ag08}
Let $N$ be an irreducible $3$-manifold such that $\pi_1(N)$
is virtually RFRS. Let $\phi\in H^1(N;\Q)$ be a non-fibered class.
There exists then a finite cover $p\colon N'\to N$ such that
$p^*(\phi)\in H^1(N';\Q)$ is quasi-fibered.
\end{theorem}

We can now also formulate the following theorem which is a
consequence of the work of  Agol \cite{Ag08,Ag12}, Liu \cite{Li11},
Wise \cite{Wi12a,Wi12b} and Przytycki--Wise \cite{PW11,PW12},
building on work of Kahn--Markovic \cite{KM12} and
Haglund--Wise \cite{HW08}.  We refer to \cite[Section~5]{AFW12}
for more background and details.

\begin{theorem}
\label{thm:akmwrfrs}
Let $N$ be an irreducible $3$-manifold with empty or toroidal boundary.
If $N$ is not a closed graph manifold, then $\pi_1(N)$ is virtually RFRS.
\end{theorem}

We now obtain the following corollary.

\begin{corollary}
\label{cor:vfibered}
Let $N$ be an irreducible $3$-manifold with empty or toroidal boundary.
If $N$  is not a closed graph manifold, then $N$ is virtually fibered.
\end{corollary}

\begin{proof}
It follows from Theorem \ref{thm:akmwrfrs} that $\pi_1(N)$
is virtually RFRS.  Since $N$ is not a graph manifold we
know that $N$ is not spherical, i.e., $\pi_1(N)$ is not finite.
We therefore see that $\pi_1(N)$ admits a finite-index
non-trivial subgroup $\pi'$ which is RFRS.

Let us denote by $N'$ the corresponding finite cover of $N$.
Since $\pi'$ is RFRS and non-trivial, it follows that
$b_1(\pi')=b_1(N')>0$.  It now follows from Theorem \ref{thm:ag08}
that $N'$, and hence $N$, admits a finite cover $M$ which is fibered.
(If $N$ is a graph manifold with non-trivial boundary, then this also follows from \cite{WY97}.)
\end{proof}

We conclude this section with the following result, which
is also an immediate consequence of Theorem \ref{thm:ag08}
(see e.g.~\cite[Section~6]{AFW12} for details).

\begin{corollary}
\label{cor:big betti}
Let $N$ be an irreducible $3$-manifold with empty or toroidal
boundary.  Suppose $N$ is neither $S^1\times D^2$, nor
$S^1\times S^1\times I$, nor finitely cover by a torus bundle. Then, for
every $k\in \N$, there is a finite cover $N'\to N$ such that
$b_1(N')\ge k$.
\end{corollary}

\section{K\"ahler $3$-manifold groups}
\label{sect:kahler3m}

After these preparations, we are now ready to prove
Theorem \ref{mainthmk}.  For the reader's convenience,
we first recall the statement of that theorem.

\begin{theorem}
\label{thm:3k}
Let $N$ be a $3$-manifold with non-empty, toroidal boundary.
If $\pi_1(N)$ is a K\"ahler group, then $N\cong S^1\times S^1\times I$.
\end{theorem}

\begin{proof}
Let $N$ be a $3$-manifold such that $\partial N$ is a non-empty
collection of tori, and assume $\pi_1(N)$ is a K\"ahler group.
As we pointed out in \S\ref{sect:kahler}, the group $\pi_1(N)$
cannot be the free product of two non-trivial groups; thus,
$N$ has to be a prime $3$-manifold.  Since $N$ has non-empty
boundary, we conclude that $N$ is in fact irreducible.

It now follows from Corollary \ref{cor:vfibered} (again using
the assumption that $N$ is not closed) that $N$ admits
a finite cover $M$ which is fibered. By Lemma \ref{lem:kahlersubgroup},
the group $\pi_1(M)$ is also a K\"ahler group. Note that the
fiber $F$ of the fibration $M\to S^1$ is a surface with boundary.
There are three cases to consider.

If $\chi(F)=1$, then $M=S^1\times D^2$.  But this is not
possible, since, as we pointed out in \S\ref{sect:kahler},
the first Betti number of a K\"ahler group is even.

If $\chi(F)=0$, then $M=S^1\times S^1\times [0,1]$,
and either $N=M$, or $N$ is the twisted $I$-bundle over the
Klein bottle.  In the latter case, $b_1(N)=1$, again contradicting
the assumption that $\pi_1(N)$ is a K\"ahler group.

Finally, if $\chi(F)<0$, then it follows from Corollary \ref{cor:mcmullen}
that $\th(M)\geq 1$. But this is not possible, since, by
Theorem \ref{thm:deltakahler}, we must have $\th(\pi_1(M))=0$.
\end{proof}

\begin{remark}
\label{rem:alt proof}
Surely there are other ways to prove this theorem.  Let us briefly
sketch an alternate approach, relying in part on some machinery
outside the scope of this paper.

It follows from the work of Agol, Wise and Przytycki--Wise
that the fundamental group of an irreducible
$3$-manifold $N$ with non-empty boundary admits a
finite-index subgroup which is a subgroup of a Coxeter
group (see \cite{AFW12} for details).  Combining this
with \cite[Theorem~A]{Py12} and \cite[Theorem~A]{BHMS02},
one can show that if $\pi_1(N)$ is a K\"ahler group, then either
$\pi$ is finite, or $\pi$ admits a finite-index subgroup which is a
non-trivial direct product of free groups (possibly infinite cyclic).
Hence, by the discussion in \S\ref{sect:kahler}, we must have
$\pi_1(N)=\Z^2$, and thus $N=S^1\times S^1\times I$.
\end{remark}

\section{Quasi--projective $3$-manifold groups}
\label{sect:qp 3m groups}

In this section, we prove Theorem \ref{mainthmqk} from the
Introduction. We will do that in several steps.

\subsection{Fibered faces in finite covers}
\label{subsec:fib fin}
We start out with the following proposition.

\begin{proposition}
\label{prop:manyfaces}
Let $N$ be an irreducible $3$-manifold with empty or toroidal boundary
which is not a graph manifold.  Then given any $k\in \N$, there exists a
finite cover $M\to N$ such that the Thurston norm ball of $M$ has at
least $k$ non-equivalent faces.
\end{proposition}

\begin{proof}
If $N$ is hyperbolic, then we know from
Corollary \ref{cor:big betti} that
 $N$  admits finite covers with arbitrarily large Betti numbers.
The proposition is then an immediate consequence of the fact that
the Thurston norm for a hyperbolic $3$-manifold is non-degenerate.
We refer to \cite[Section~8.4]{AFW12} for a proof of the proposition
in the non-hyperbolic case.
\end{proof}

The following is now a well-known consequence of
Proposition \ref{prop:manyfaces} and Theorem \ref{thm:ag08}.

\begin{theorem}
\label{thm:manyfiberedfaces}
Let $N$ be an irreducible $3$-manifold with empty or toroidal boundary
which is not a graph manifold.  Then given any $k\in \N$, there exists
a finite cover $N'\to N$ such that the Thurston norm ball of $N'$ has
at least $k$ non-equivalent fibered faces.
\end{theorem}

The proof is basically identical with the proof of \cite[Theorem~7.2]{Ag08}.
We provide a quick outline of the proof for the reader's convenience.

\begin{proof}
We pick classes $\phi_1,\dots,\phi_k$ in $H^1(N;\Q)$ which lie in $k$
inequivalent faces.
Note that $\pi_1(N)$ is virtually RFRS by Theorem \ref{thm:akmwrfrs}.
For each $i=1,\dots,k$ we can therefore apply Theorem \ref{thm:ag08}
to the class $\phi_i$, obtaining a finite cover $\wti{N}_i\to N$ such
that the pull-back of $\phi_i$ is quasi-fibered.

We now denote by $p\colon M\to N$ the cover corresponding to
the subgroup $\bigcap_{i=1}^{k} \pi_1(\wti{N}_i)$.  By Corollary \ref{cor:pull qf},
the cohomology classes $p^*(\phi_1),\dots, p^*(\phi_k)$ lie on closures of
inequivalent fibered faces of $M$.  Hence, $M$ has at least $k$
inequivalent fibered faces.
\end{proof}

\subsection{Irreducible $3$-manifolds which are not graph manifolds}
\label{subsec:not gm}

Next, we upgrade the statement about the Thurston unit ball
from  Theorem \ref{thm:manyfiberedfaces} to a statement about
the thickness of the Alexander ball.

\begin{theorem}
\label{thm:fatalexander}
Let $N$ be an irreducible $3$-manifold with empty or
toroidal boundary, and suppose $N$ is not a graph manifold.
There exists then a finite cover $N'\to N$ with
$\th(N')\geq 2$ and $b_1(N')\geq 3$.
\end{theorem}

\begin{proof}
By Corollary \ref{cor:big betti}, $N$ admits
finite covers with arbitrarily large first Betti numbers.
We can thus without loss of generality assume that
$b_1(N)\geq 3$.

By Theorem \ref{thm:manyfiberedfaces}, there exists a finite cover
$N'\to N$ such that the Thurston norm ball of $N'$ has at
least $2$ non-equivalent fibered faces. A basic
transfer argument shows that $b_1(N')\geq b_1(N)\geq 3$.
By Corollary  \ref{cor:mcmullen}, we have that $\th(N')\geq 2$,
and we are done.
\end{proof}

We can now prove Theorem \ref{mainthmqk} in the case
that $N$ is irreducible.

\begin{theorem}
\label{thm:not graph mfd}
Let $N$ be an irreducible $3$-manifold with empty or toroidal
boundary.  If $N$ is not a graph manifold, then $\pi_1(N)$ is
not a \qp group.
\end{theorem}

\begin{proof}
Suppose $\pi_1(N)$ is a \qp group.  By Theorem \ref{thm:fatalexander},
there exists a finite cover $N'\to N$ with $\th(N')\geq 2$ and
$b_1(N')\geq 3$.  By Lemma \ref{lem:qpsubgroup},
$\pi_1(N')$ is also a \qp group, which implies by
Theorem \ref{thm:deltaquasikahler} that either
$b_1(N')=2$ or $\th(N')\leq 1$.  We have thus
arrived at a contradiction.
\end{proof}

Let us draw an immediate corollary.

\begin{corollary}
\label{cor:hyp}
If $N$ is a hyperbolic $3$-manifold with empty or toroidal
boundary, then $\pi_1(N)$ is not a \qp group.
\end{corollary}

\subsection{Reducible $3$-manifolds}
\label{section:nonprimeqk}
We turn now to the proof of Theorem \ref{mainthmqk}
for non-prime $3$-manifolds. We start out with a
straightforward observation.

Let $N$ be a $3$-manifold, and let $N=N_1\# \cdots \# N_s$
be its prime decomposition.   Set $H=H_1(N;\Z)/\tor$ and
$H_i=H_1(N_i;\Z)/\tor$. The Mayer--Vietoris sequence
with $\Z$-coefficients shows that the inclusion maps
induce an isomorphism
$H_1\oplus \cdots \oplus H_s\xrightarrow{\cong}H$.

\begin{lemma}
\label{lem:deltamultiply}
Let $N$ be a $3$-manifold which admits a decomposition
$N=N_1\# N_2\# N_3$, and identify $H$ with
$H_1\oplus H_2\oplus H_3$ as above.
 Suppose that $H_1$ and $H_2$ are non-zero. For $i=1,2$
denote by $r_i$ the rank of $H_1(N_i;\Z[H_i])$, and denote
by $r$ the rank of $H_1(N;\Z[H])$. Then there exists a non-zero
polynomial $f\in \Z[H]$ such that
\[
\Delta_{N}^{r}=\Delta_{N_1}^{r_1}\cdot \Delta_{N_2}^{r_2}\cdot f\in \Z[H]=\Z[H_1\oplus H_2\oplus H_3].
\]
In particular,
\[
\th(N)\geq \th(N_1)+\th(N_2).
\]

\end{lemma}

\begin{proof}
The Mayer--Vietoris sequence for
$N=N_1\# N_2\# N_3$ with coefficients in $\Z[H]$
yields an exact sequence,
\begin{equation}
\label{eq:mv}
\xymatrixcolsep{16pt}
\xymatrix{
0 \ar[r]& H_1(N_1;\Z[H])\oplus H_1(N_2;\Z[H])\oplus
H_1(N_3;\Z[H]) \ar[r]& H_1(N;\Z[H]) \ar[r]& \Z[H]^{2}},
\end{equation}
where the last term corresponds to the $0$-th homology
of the two gluing spheres.  From this, we
get a monomorphism
\begin{equation}
\label{eq:torcong}
\xymatrixcolsep{18pt}
\xymatrix{\tor_{\Z[H]}H_1(N_1;\Z[H])\oplus
\tor_{\Z[H]}H_1(N_2;\Z[H])\ar[r]&
\tor_{\Z[H]}H_1(N;\Z[H])}.
\end{equation}

Note that the alternating product of orders in an exact sequence
of torsion modules is $1$ (see e.g.~\cite[p.~57~and~60]{Hil02}).
Consequently, there exists a non-zero polynomial
$f\in \Z[H_3]$ such that
\begin{equation}
\label{eq:ordtor}
\begin{split}
\ord_{\Z[H]}(\tor_{\Z[H]}H_1(N;\Z[H]))=\ord_{\Z[H]}(\tor_{\Z[H]} H_1(N_1;\Z[H]))\cdot \hspace{0.4in} \\
\ord_{\Z[H]}(\tor_{\Z[H]}H_1(N_2;\Z[H]))\cdot f.
\end{split}
\end{equation}

Now note that $H_1(N_i;\Z[H])\cong H_1(N_i;\Z[H_i])
\otimes_{\Z[H_i]}\Z[H]$, for $i=1,2,3$.  Applying
Lemma \ref{lem:thicknessadds} and formula \eqref{equ:tu01},
it follows that
\begin{equation}
\label{eq:delnr}
 \Delta_{N}^{r}=\Delta_{N_1}^{r_1}\cdot \Delta_{N_2}^{r_2}\cdot f\in \Z[H].
\end{equation}
Using Lemma \ref{lem:thicknessadds} again, we conclude that
\begin{equation}
\label{eq:thsum}
\th(N)= \th(N_1)+\th(N_2)+\th(f)\geq \th(N_1)+\th(N_2),
\end{equation}
and this completes the proof.
\end{proof}

We now obtain the following result, which in particular implies
Theorem \ref{mainthmqk} for non-prime $3$-manifolds.

\begin{theorem}
\label{thm:qknonprime}
Let $N$ be a $3$-manifold with empty or toroidal boundary
which is not prime.  If $\pi_1(N)$ is a \qp group, then all
prime components of $N$ are closed graph manifolds
which are not virtually fibered.
\end{theorem}

\begin{proof}
In light of Corollary \ref{cor:vfibered}, it suffices to prove
the following:  If $N$ is a non-prime $3$-manifold with
empty or toroidal boundary, and $N$ has at least one
prime component which is virtually fibered, then
$\pi_1(N)$ not a \qp group.

Recall from Lemma \ref{lem:qpsubgroup} that a finite-index
subgroup of a \qp group is again \qsp. Thus,
after going to a finite cover if necessary, we may
assume that $N$ admits a decomposition
$N=N_1\# N_2$, where $N_1$ is fibered and $N_2\ne S^3$.

First suppose that $N_1$ is a Sol-manifold and $N_2=\R P^3$.
A straightforward calculation shows that
\begin{equation}
\label{eq:deln1}
\Delta_{N_1}=2(t-\l)(t-\l^{-1})\in \Z[H_1(N;\Z)/\tor]=\Z[\Z]=\zt,
\end{equation}
where $\l\ne \pm 1$ is a real number.  (The factor of $2$ comes
from the order of $H_1(\R P^3;\Z)=\Z_2$.)
It now follows from Theorem \ref{thm:deltaquasikahler}\eqref{q2}
and Lemma \ref{lem:deltamultiply} that $\pi_1(N)$ is not a \qp group.

Next, suppose that $N_1$ is not a Sol-manifold. If $N_1$
is covered by a torus bundle, then the fact that $N_1$ is not
a Sol-manifold implies that $N_1$ admits a finite cover with
$b_1\geq 2$ (see e.g.~\cite{AFW12} for details).  If $N_1$ is not
covered by a torus bundle, then it follows from Corollary \ref{cor:big betti}
that $N_1$ admits a finite cover with $b_1\geq 2$.
We can therefore without loss of generality assume that $b_1(N)\geq 2$.

Since $N_2$ is not the $3$-sphere, it follows from the Poincar\'e
Conjecture that $\pi_1(N_2)$ is non-trivial.  Furthermore, it follows
from the Geometrization Theorem that  $\pi_1(N_2)$ is residually
finite (see \cite{He87}). Therefore, there exists an epimorphism
$\pi_1(N_2)\surj G$ onto a non-trivial finite group.
Let $M\to N$ be the cover corresponding to
the epimorphism
\begin{equation}
\label{eq:epi}
\xymatrixcolsep{18pt}
\xymatrix{
\pi_1(N)=\pi_1(N_1)*\pi_1(N_2)\ar[r]& \pi_1(N_2)\ar[r]& G
}.
\end{equation}

Note that the above homomorphism is trivial on $\pi_1(N_1)$;
hence, $M$ contains at least $\abs{G}$ prime components
which are diffeomorphic to $N_1$.  Also note that $b_1(M)$ is
the sum of the Betti numbers of the prime components of $M$.
By the above, we see that $b_1(M)\ge 4$.
Furthermore, it follows from Lemma \ref{lem:deltamultiply} that
$\th(M)$ is the sum of the thicknesses of the prime components
of $M$.  Since $\th(N_1)\geq 1$, we conclude that
\begin{equation}
\label{eq:thmg}
\th(M)\geq \abs{G}\cdot \th(N_1)\geq \abs{G}\geq 2.
\end{equation}
It now follows from  Theorem \ref{thm:deltaquasikahler}
that $\pi_1(M)$ is not a \qp group.

Finally, suppose that $N_1$ is a Sol-manifold but that $N_2\ne \R P^3$.
Then it follows from the Geometrization Theorem that $\pi_1(N_2)$ has
more than two elements. Since $\pi_1(N_2)$ is residually finite there
exists an epimorphism $\pi_1(N_2)\surj G$ onto a  finite group with at
least three elements.  Again, let $M\to N$ be the cover
corresponding to the epimorphism \eqref{eq:epi}.
As above, we note that $M$ has at least three prime components
homeomorphic to $N_1$, which then implies that $b_1(M)\geq 3$ and
$\th(M)\geq 3$. Applying again Theorem \ref{thm:deltaquasikahler},
we conclude that $\pi_1(M)$ is not a \qp group.
\end{proof}

\section{Open questions}
\label{sect:questions}

We conclude this paper with some open  questions.
First recall that a K\"ahler group is never the free
product of two non-trivial groups.  On the other
hand, free groups are \qp groups. The following
question is open to the best of our knowledge.

\begin{question}
\label{quest:free prod}
Suppose $A$ and $B$ are groups, such that the free product
$A*B$ is a \qp group.  Does it follow that $A$ and $B$
are already \qp groups?
\end{question}

We showed that if $N$ is an irreducible $3$-manifold with
empty or toroidal boundary such that $\pi_1(N)$ is a
\qp group, then $N$ is a graph manifold.  Not all
graph manifold groups are \qp groups, though,
as the next example (which we already encountered in the
proof of Theorem \ref{thm:qknonprime}) shows.

\begin{example}
\label{ex:sol}
Suppose $N$ is a torus bundle whose monodromy
has eigenvalues $\l$ and $\l^{-1}$, for some real number
$\l>1$.  Then $N$ is a Sol manifold, and thus a graph manifold.
On the other hand, $b_1(N)=1$, and the Alexander polynomial
$\Delta^1_N$ equals $(t-\l)(t-\l^{-1})$.  Hence, by
Theorem \ref{thm:deltaquasikahler}\eqref{q2}, the
fundamental group of $N$ is not a \qp group.
\end{example}

\begin{question}
\label{quest:graph qk}
For which graph manifolds is the fundamental group a \qp group?
\end{question}

Finally, the case of connected sums of graph manifolds has
also not been completely settled.  In light of the results in
\S\ref{section:nonprimeqk}, we venture the following conjecture.

\begin{conjecture}
\label{conj:conn sum}
Let $N$ be a compact $3$-manifold with empty or toroidal
boundary.  If $\pi_1(N)$ is a \qp group and if $N$ is not prime,
then $N$ is the connected sum of spherical $3$-manifolds
and manifolds which are either diffeomorphic to
$S^1\times D^2$, $S^1\times S^1\times [0,1]$,
or the $3$-torus.
\end{conjecture}

Note that the groups of the prime $3$-manifolds listed above
are either finite groups, and thus projective, or finitely generated
free abelian groups, and thus quasi-projective.

We finish with one more question.

\begin{question}
\label{quest:qp-qk}
Can the statements we prove here regarding \qsp, $3$-manifold groups
be extended to the analogous statements for \qk, $3$-manifold groups?
\end{question}

If one could prove Theorem \ref{thm:arapura} for an arbitrary
quasi-K\"ahler manifold, the answer to this question would be yes:
all our remaining arguments would then go through in this
wider generality.

\begin{ack}
The second author thanks the University of Warwick,
the Max Planck Institute for Mathematics in Bonn,
and the University of Sydney for their support
and hospitality while part of this work was carried out.
\end{ack}

\newcommand{\arxiv}[1]
{\texttt{\href{http://arxiv.org/abs/#1}{arXiv:#1}}}
\newcommand{\arx}[1]
{\texttt{\href{http://arxiv.org/abs/#1}{arXiv:}}
\texttt{\href{http://arxiv.org/abs/#1}{#1}}}
\newcommand{\doi}[1]
{\texttt{\href{http://dx.doi.org/#1}{doi:#1}}}
\renewcommand{\MR}[1]
{\href{http://www.ams.org/mathscinet-getitem?mr=#1}{MR#1}}
\newcommand{\MRh}[2]
{\href{http://www.ams.org/mathscinet-getitem?mr=#1}{MR#1 (#2)}}

\end{document}